\documentclass[11pt]{amsart}
\usepackage[margin=3.3cm]{geometry}
\usepackage[final]{weilmac} 
\operators{Stab,Sel,Gal,Irr,Disc,Jac,diag,GSp,Fil,Sp}
\calsymbols{c}{C,K,F,X}

\def\Het{H_{\text{\'et}}}
\def\Ql{\Q_\ell}
\def\res#1{{\F}_{#1}}
\def\resbar#1{\bar{\F}_{#1}}

\begin{document}

\title{Character formula for Weil representations in terms of Frobenius traces}

\author{Tim Dokchitser}
\address{Department of Mathematics, University Walk, Bristol BS8 1TW, UK}
\email{tim.dokchitser@bristol.ac.uk}

\author{Vladimir Dokchitser}
\address{University College London, London WC1H 0AY, UK}
\email{v.dokchitser@ucl.ac.uk}


\begin{abstract}
It is known that the \'etale cohomology of a potentially good abelian variety over a local field $K$ is 
determined by its Euler factors over the extensions of $K$. We extend this to all abelian varieties,
show that it is enough to take extensions where $A$ is semistable, and give a uniform version 
over $p$-adic fields where 
the extensions are the same for all abelian varieties of a given dimension.
The results are explicit, and apply to a wide class of Weil-Deligne representations.
\end{abstract}

\maketitle




\marginpar{Rewrite to stress bad reduction as easy as good reduction for curves and abelian varieties}

Ramified $l$-adic local Galois representations tend to be unwelcome. 
We give an explicit character formula which reconstructs such a representation from its Frobenius traces over extensions 
where it is unramified. 
We also record some consequences for \'etale cohomology of abelian varieties 
and other $l$-adic representations. 
The paper is a number-theoretic counterpart of \cite{clifford}, which concerns representations of finite groups.

Our setting is as follows. Let $K$ be a non-archimedean local field, with residue field $\res K$,
separable closure $K^s$ and algebraic closure~$\bar K$.
We write $K^{nr}$ for the maximal unramified extension of $K$ in $K^s$,
$I_K=I_{K^s/K}$ for the absolute inertia group and 
$\Frob_K\in G_K=\Gal(K^s/K)$ for an arithmetic Frobenius element.
Recall that the Weil group $W_K$ is generated by $I_K$ and $\Frob_K$; its topology is
the usual profinite one on $I_K$ and discrete on $W_K/I_K\iso\Z$. A Weil representation over $K$ is a continuous
finite-dimensional complex representation of $W_K$. 
We refer the reader to \cite{TatN} for the background on Weil
and Weil-Deligne representations. 


We begin with an analogue of the character formula \cite[Cor. 9]{clifford} for Weil representations:

\begin{theorem}
\label{X16}
Let $F/K$ be a finite Galois extension of local fields and $\rho$ 
a semisimple Weil representation over $K$ that factors through $\Gal(F^{nr}/K)$.
Let $\{\rho_i\}_{i\in\Lambda}$ be a set of irreducible 
representations of $\Gal(F/K)$, with exactly one in each set of unramified twists%
\footnote{i.e. $\rho_i\ne \rho_j\tensor$(1-dim) unramified, and each irreducible of 
$\Gal(F/K)$ is $\rho_i\tensor$(1-dim unramified) for some~$i$}. 
Write $I<\Gal(F/K)$ for the inertia group and $m_i = \langle \Res_I\rho_i, \Res_I\rho_i\rangle$.
Then
\begin{itemize}
\item[(i)]
There are unramified representations $\Psi_i$ such that
$
  \rho\iso\bigoplus_{i\in\Lambda} \rho_i\tensor\Psi_i.
$
\item[(ii)]
Fix $i\in\Lambda$. 
For every $d>0$,
$
 \tr \Psi_i(\Frob_K^{dm_i})=\frac{1}{|I|\,m_i} \sum_{L} \overline{\tr \rho_i(\Frob_L)}\,\tr \rho(\Frob_L),
$
where $L$ ranges over extensions of $K$ in $\smash{F^{nr}}$ of ramification degree $|I|$ 
and residue degree $dm_i$.
Note that $F^{nr}/L$ is unramified, so $\rho(\Frob_L), \rho_i(\Frob_L)$ are \hbox{well-defined.}

%
\item[(iii)] 
The twist $\rho_i\tensor\Psi_i$ is uniquely determined by {\rm (ii)}.

\item[(iv)] 
Concretely, 
suppose
\hbox{$\dim\Psi_i\!\le\! N$}. There is a unique $0\!\le\!n\!\le\!N$ and $\lambda_1,...,\lambda_n\in\C^\times$ 
such that 
$\sum_k\lambda_k^d=\tr \Psi_i(\Frob_K^{dm_i})$ for $d=1,...,N$.
Then 
$\Psi_i(\Frob_K)$
has eigenvalues $\sqrt[m_i]{\lambda_1},...,\sqrt[m_i]{\lambda_n}$
for some choice of the roots, and
%
$\rho_i\tensor\Psi_i$ is independent of this choice.
\end{itemize}
\end{theorem}

Let $\ell$ be a prime different from the residue characteristic of $K$.
As an application, we deduce that the $\ell$-adic representation $\rho_\ell=\Het^1(A_{\bar K},\Ql)$ 
of an abelian variety $A/K$ is determined by the traces $\Tr\rho^{I_L}_\ell(\Frob_L)$ over the fields $L$ where $A$ 
is semistable. 
(Here, as usual, $\rho_\ell^{I_L}$ is the maximal subrepresentation unramified over $L$.)
In fact, when $K$ has characteristic 0, we can choose a universal finite list 
of such fields to control all abelian varieties of a fixed dimension.


\begin{theorem}
\label{abmain}
Let $g\ge 0$, and suppose $\vchar K=0$. There is a finite Galois extension $K_g/K$ such that
\begin{enumerate}
\item[(i)]
Every $g$-dimensional abelian variety $A/K$ has semistable reduction over $K_g$.
\item[(ii)]
The $G_K$-representation
$\rho_\ell=\Het^1(A_{\bar K},\Ql)$ is uniquely determined by the traces\\ $\Tr\rho^{I_L}_\ell(\Frob_L)$ over subfields 
$K\subset L\subset K_g$ for which $K_g/L$ is unramified.
\item[(iii)]
When $g\ge 2$, for every smooth projective curve $X/K$ of genus $g$ the $G_K$-representation
$\Het^1(X_{\bar K},\Ql)$ is uniquely determined by point counts $\#\bar \cX_L(\res L)$
with $L$ as in (ii).
Here $\cX_L$ is the minimal regular model of $X/L$ with special fibre $\bar \cX_L/\res L$.
\end{enumerate}
\end{theorem}

%


\noindent
We caution the reader that the result is false when $\vchar K>0$, see Remark \ref{oops}.

We now return to the case of arbitrary characteristic.
Theorem \ref{X16} allows one to explicitly reconstruct local Galois 
representations from the restrictions to subgroups where they become semistable
or unramified (extending \cite{weil}). For example, suppose $X/K$ is a curve with `very bad' reduction.
Then we can identify the Galois representation 
$\rho_l=\Het^1(X_{\bar K},\Ql)$, provided we know a Galois extension $F/K$ where $X$ becomes semistable,
and we can count points on the reduction of $X$ in extensions of $K$ where $X$ is semistable.
We restore the decomposition of $\rho_l$ into irreducibles via a weighted average of Frobenius traces 
as in Theorem \ref{X16} (ii), which can be computed as point counts on the reduced curves. 
We give a numerical example at the end of this paper, which illustrates this in detail.
This style of argument has been used for specific Galois groups (see 
\cite[\S\S 3-4]{weil}, \cite[Thm 7.3]{newton}, \cite{nir3}, \cite[App. A]{KP}), and Theorem \ref{X16}
gives a universal character formula which applies for arbitrary groups.


%
%
%
%

Here are some theoretical consequences of Theorem \ref{X16} for general representations.
The following corollary is essentially due to Saito \cite[Lemma 1(1)]{Sai}; see also Remark \ref{Che}.

\begin{corollary}
\label{repmain}
Every semisimple Weil representation $\rho$ over $K$
is uniquely determined by the 
traces
$
  \tr \rho(\Frob_L),
$
where $L$ varies over finite separable extensions of $K$
over which $\rho$ is unramified.
\end{corollary}


\begin{notation*}
Recall that a Frobenius semisimple Weil-Deligne representation 
$\rho=(\rho_{\rm Weil},N)$ over $K$ can be decomposed as
$
  \rho = \bigoplus_{n\ge 1} \rho_n \tensor \Sp_n, 
$
where $\rho_n$ are semisimple Weil representations (i.e. have $N=0$), 
and $\Sp_n$ is the $n$-dimensional special representation \cite[4.1.4--4.1.5]{TatN}.
In particular,  $\rho^{N=0}:=\ker N= \bigoplus \rho_n$. 

We say that $\rho$ 
is \emph{weight-monodromy compatible} if 
the eigenvalues of $\Frob_K$ on $\rho_i$ are Weil numbers of absolute value 
\marginpar{Explain trace of Frobenius on inertia invariants
  as a term of the characteristic polynomial}
$q^{\frac{n-1}2}$, where $q=|\res K|$. 
We use the same term for $\ell$-adic representations when their associated Weil-Deligne representation
is weight-monodromy compatible.
\end{notation*}

For every proper smooth variety $X/K$ and $0\le i\le2\dim X$,
$\Het^i(X_{\bar K},\Ql)(\frac i2)$ is conjectured to be weight-monodromy 
compatible \cite{DelH1}, and this is often known (see \cite[\S1]{Sch} 
for a summary).
Because a weight-monodromy compatible representation is determined by 
$\rho^{N=0}$, from Corollary \ref{repmain} we get


\begin{corollary}
\label{wdmain}
Frobenius-semisimple Weil-Deligne representations $\rho=(\rho_{\rm Weil},N)$ over $K$ 
satisfying the weight-monodromy compatibility are
uniquely determined by the 
traces $\tr\rho^{N=0}(\Frob_L)$,
where $L$ varies over finite separable extensions of $K$ over which $\rho_{\rm Weil}$ is unramified.
\end{corollary}

\begin{corollary}
\label{nonabmain}
Let $X/K$ be a proper smooth variety and $0\le i\le 2\dim X$. 
If the $G_K$-representation $\rho_\ell=\Het^i(X_{\bar K},\Ql)(\frac i2)$ is Frobenius semisimple and 
weight-monodromy compatible, then it is uniquely determined by 
 $\Tr\rho_\ell^{I_L}(\Frob_L)$ 
for finite separable extensions $L/K$ for which $I_L$ acts unipotently.
In particular, if these traces are independent of $\ell$, then so is
the Weil-Deligne representation associated to $\Het^i(X_{\bar K},\Ql)$.
\marginpar{$\rho_\ell^{ss}$ is unramified?}
\end{corollary}


For semistable abelian varieties these traces are known to be independent of 
$\ell$ \cite[IX, Thm~4.3]{SGA7I}, and
we get the well-known
independence of $\ell$ for Weil-Deligne representations of general 
abelian varieties (see \cite[Ex. 8.10]{DelC}, 
\cite[Rem. 2.4.6]{Fon94} and 
\cite[Prop. 2.8.1]{BCGP} for a proof): 

\begin{corollary}
\label{independence}
For an abelian variety $A/K$, the Weil--Deligne representation associated 
to $\Het^1(A_{\bar K}, \Ql)$ is independent of $\ell$.
\end{corollary}

Finally, we also have a version of Theorem \ref{abmain} for 
semisimple $\ell$-adic representations:

\def\KFn{K_{\cF,n}}
\def\KFnp{E}

\begin{theorem}
\label{genunifield}
Suppose $\vchar K=0$. Let $n\ge 0$ and $\mathcal{F}/\Ql$ a finite extension. There is a finite Galois extension $\KFn/K$ such that all 
continuous semisimple representations $\rho_\ell: G_K\to \GL_n(\mathcal{F})$ are 
\begin{enumerate}
\item 
unramified over subfields $K\subset L\subset \KFn$ for which
$\KFn/L$ is unramified;

\item
uniquely determined by the traces $\tr\rho_\ell(\Frob_L)$ for $L$ as in (1).
\end{enumerate}
Moreover, any finite extension $\KFnp/\KFn$, with $\KFnp$ Galois over $K$, has this property.
\end{theorem}

Again, because a Frobenius semisimple weight-monodromy compatible representation 
is determined by its maximal semisimple subrepresentation, we deduce
%

\begin{corollary}
\label{genunifield2}
Suppose $\vchar K=0$.
All Frobenius semisimple weight-monodromy compatible representations $\rho_\ell\colon G_K\to \GL_n(\mathcal{F})$
are uniquely determined by the traces $\Tr\rho_\ell^{I_L}(\Frob_L)$ 
with $L$ and $\KFn$ (or $\KFnp$) as in the theorem.
\end{corollary}

\begin{remark}
\label{Che}
Our strategy of proving that Weil representations are determined by their 
Frobenius traces over extensions where they are unramified is known, 
and goes back at least to Chebotarev. He used the fact that after adjoining enough 
roots of unity, a finite extension of number fields has abundance of cyclotomic extensions 
inside it to reduce his density theorem to Hecke's cyclotomic case
(see e.g. \cite{LS}). Similarly, 
Corollary \ref{repmain} relies on the abudance of unramified extensions inside $F^{nr}/K$
for a finite Galois extension $F/K$. 
What is new is Theorem \ref{X16}, which makes this explicit in practice.

As mentioned above, Corollary \ref{repmain} is due to Saito \cite[Lemma 1(1)]{Sai}. His
assumption $n(\sigma)\ge 0$ is slightly stronger but the proof only uses 
$n(\sigma)>0$, which is equivalent to that in Corollary \ref{repmain}. 
The proof of \cite[Lemma 1(1)]{Sai} does not explain why $m$ is a function of the traces,
and the field cut out by $\sigma$ is not a finite extension of $K$. 
However, both of these issues can be circumvented with an argument similar to the proof of 
Corollary \ref{repmain} below, by comparing two representations instead of taking one.
\end{remark}

\section*{Proof of Theorem \ref{X16}}

(i) This is standard; see e.g. \cite[Lemma 2.3]{weil} and its proof (taking $F$ as in the theorem).

(ii) By linearity in $\rho$, it suffices to prove this when $\rho=\rho_i\tensor\Psi$, with $\Psi$ 
unramified 1-dimensional.
Because the formula is also invariant under twisting by unramified 1-dimensional characters, 
we may assume $\Psi=\triv$. 
In particular $\Psi$ factors through $\Gal(F/K)$, and this case is proved in \cite[Cor 9 (iii)]{clifford}.
Note that ranging over $L$ is the same as ranging over $g\in \Frob_K^{dm_i}\!I_{F^{nr}/K}$ under the correspondence $L\mapsto \Frob_L$ and $g\mapsto$ $g$-invariants of~$F^{nr}$.

(iii),(iv) 
This is essentially \cite[Cor 9 (iv)]{clifford}. 
It is stated there for finite groups, but the proof is the same: 
(iii) and (iv) follow formally from (ii) plus the fact that 
$\rho_i$ is invariant under twisting by unramified characters of order $m_i$,
which is \cite[Lemma 5]{clifford}.
\hfill\qed

\section*{Proof of Corollary \ref{repmain}}

Let $\rho_1, \rho_2$ be two non-isomorphic 
semisimple Weil representations over $K$.
They have finite inertia images, and so there is 
a finite Galois extension $F/K$ such that both $\rho_1$ and $\rho_2$
factor through $\Gal(F^{nr}/K)$. By
Theorem \ref{X16}, $\rho_1$ and $\rho_2$
can be distinguished from one another by their 
Frobenius traces over extensions $L/K$ inside $F^{nr}$
over which they are unramified.
%
%
%
%
%

%
%
%

%


\section*{Proof of Theorem \ref{genunifield}}

Fix $n$ and $\cF$ as in the theorem.
First consider continuous representations $\tau: G_K\to\GL_n(\cF)$ with finite image. 
After taking a $\tau(G_K)$-invariant $O_\cF$-lattice, we 
may assume that $\tau$ lands in $\GL_n(O_\cF)$. 
Recall that $\GL_n(O_\cF)$ has a compact open subgroup $U$
with no elements of finite order (e.g. using exp and log as in \cite[Appendix]{ST}).
\marginpar{MO 287553}
\marginpar{or by comparing characteristic polynomials}
%
%
Therefore $|\tau(G_K)|\le (\GL_n(O_\cF):U)$, and $\tau$ factors thorough a Galois extension of $K$ of 
at most this degree. As $\vchar K=0$, there are only finitely many separable extensions of $K$ of a given degree
by a theorem of Krasner \cite{Kr}, and so%
\marginpar{Outsource to a remark on Thm 3, and note it applies to virtual representations as well}
there is a finite Galois extension $F/K$, such that every $\tau$ with finite image factors through $\Gal(F/K)$.
It follows that every continuous semisimple representation 
$\rho: G_K\to \GL_n(\cF)$ factors through $\Gal(F^{nr}/K)$.%

Now apply Theorem \ref{X16}. Let $\KFn$ be the compositum of all $L/K$ in $F^{nr}$ with
ramification degree $e_{L/K}=e_{F/K}$ and residue degree over $K$ bounded by $n^3$. By the theorem
(parts (ii)-(iv) and using that $dm_i\le n^3$ in (iv)), Frobenius traces over these fields 
determine $\rho_\C: W_K\to \GL_n(\C)$ for some chosen embedding $\cF\injects \C$.
Since $W_K\subset G_K$ is dense and $\rho_\C$ is continuous, it is determined on all of $G_K$ as well.
Therefore $\rho$ itself is uniquely determined as a representation to $\GL_n(\cF)$.
The same is true for any finite extension $\KFnp$ of $\KFn$, Galois over~$K$.
\hfill\qed
%
%

%
%

%
%

\section*{Proof of Theorem \ref{abmain}}

We will prove a slightly stronger statement:

\begin{theorem}
\label{abmain2}
Suppose $\vchar K=0$.
Fix $g, c\in \N$ and a prime $\ell$ different from the residue characteristic of $K$. 
There is a finite Galois extension $K_{\ell,g,c}/K$ such that
\begin{enumerate}
\item[(i)]
Every $g$-dimensional abelian variety $A/K$ has semistable reduction over $K_{\ell,g,c}$.
\item[(ii)]
The $G_K$-representation
$\rho_\ell=\Het^1(A_{\bar K},\Ql)$ is uniquely determined by 
$\Tr\rho^{I_L}_\ell(\Frob_L)$ for subfields 
$K\subset L\subset K_{\ell,g,c}$ with $|\res L|>c$ and $K_{\ell,g,c}/L$ unramified.
\item[(iii)]
When $g\ge 2$, for every smooth projective curve $X/K$ of genus $g$, the $G_K$-representation
$\Het^1(X_{\bar K},\Ql)$ is uniquely determined by point counts $\#\bar \cX_L(\res L)$
with $L$ as in (ii).
Here $\cX_L$ is the minimal regular model of $X/L$ with special fibre $\bar \cX_L/\res L$.
\end{enumerate}
\end{theorem}

\begin{proof}
(i), (ii) Suppose $c=1$. Recall that $\Het^1(A_{\bar K},\Ql)$ is 
Frobenius semisimple and weight-monodromy compatible \cite[IX]{SGA7I}. For $L/K$ finite, 
$I_L$ acts unipotently on it if and only if $A/L$ is semistable 
\cite[IX, 3.5/3.8]{SGA7I}, \cite[7.4.6]{BLR}.
%
%
By Theorem \ref{genunifield} and Corollary \ref{genunifield2} for $n=2g$ and $\cF=\Ql$,
any choice of $K_{\ell,g,1}$ that contains $\KFn$ satisfies (i) and (ii).

Now let $c$ be arbitrary, and choose a prime $d\ne\ell$ so that $[\Ql(\zeta_{d}):\Ql]>4g^2$ and $|\res K|^d>c$. 
(All but finitely many primes satisfy these conditions.)
Let $K'$ be the unramified extension of $K$ of degree $d$. 
Construct $K'_{\ell,g,1}$ by the $c=1$ case over $K'$, enlarging it if necessary to contain~$\KFn$.
We claim that $K_{\ell,g,c}=K'_{\ell,g,1}$ satisfies (i) and (ii). 
By construction, (i) holds.

Suppose $A/K$ is an abelian variety. 
Frobenius traces over subfields $K'\subset L\subset K'_{\ell,g,1}$  
as in (ii) determine the restriction of $\rho_\ell$ to $K'$ uniquely
by the properties of $\KFn$. Such $L$ have 
$|\res L|\ge |\res{K'}|=|\res K|^d>c$, so it suffices to show that this restriction determines~$\rho$.

By Theorem \ref{genunifield} and Corollary \ref{genunifield2} again, 
it suffices to reconstruct Frobenius traces 
over subfields $K\subset L\subset K_{\cF,n}$ with $K_{\cF,n}/L$ unramified.
If $L$ contains $K'$ then we are done, as we have already reconstructed the restriction of $\rho_\ell$ to $K'$. 
Suppose $L\not\supset K'$, and let $\alpha_1,...,\alpha_m$ ($m\le 2g$) be the eigenvalues of $\rho^{I_L}(\Frob_L)$. 
As $LK'/L$ has degree $d$ (as $d$ is prime), we know $\alpha_1^d,...,\alpha_m^d$,
the Frobenius eigenvalues over $LK'$ from the restriction to~$K'$. 

Note that $\alpha_1$ generates an extension of $\Q_l$ of degree $\le 2g$. Now, $\alpha_1^d$ has a unique $d$th root 
that generates an extension of $\Ql$ of degree $\le 2g$, namely $\alpha_1$. Indeed, if 
$\alpha_1\zeta_d^i$ also generates such an extension for $i\ne 0\bmod d$, then 
$$
  [\Ql(\alpha_1,\zeta_d):\Ql] = [\Ql(\alpha_1,\alpha_1\zeta_d^i):\Ql] \le (2g)^2 = 4g^2,
$$
contradicting $[\Ql(\zeta_d):\Ql]>4g^2$. Therefore $\alpha_1$ is indeed 
determined by $\alpha_1^d$, and similarly for all the other $\alpha_i$, as required.


(iii) 
We claim that any $K_{\ell,g,1}$ that satisfies (i) and (ii) also satisfies
(iii), provided $c>16g^2$ (which we may clearly assume).

Let $A$ be the Jacobian of $X$. Recall that there is a canonical isomorphism
\marginpar{Reference?}
$$
  \Het^1(X_{\bar K},\Q_l) \iso \Het^1(A_{\bar K},\Q_l).
$$
By (ii), this representation, say $\rho_\ell$, is determined by the traces $\Tr\rho_\ell^{I_L}(\Frob_L)$
for $L$ as in (ii). By (i), over these fields $A$ is semistable, hence so is $X$ 
as $g\ge 2$ \cite[Thm. 1.2]{DM}. Therefore
$$
  \Het^1(X_{\bar K},\Q_l)^{I_L} \iso \Het^1((\bar\cX_L)_{\resbar L},\Q_l)
$$
as $\Gal(\resbar L/\res L)$-modules, see e.g. \cite[Thm. B.1]{newton}.
By the Grothendieck-Lefschetz trace formula \cite[p. 86, Thm 3.1]{SGA412},
$$ 
  \#\bar\cX_L(\res L) = t_0 - t_1 + t_2, 
$$
where $t_i$ is the trace of the Frobenius automorphism on $\Het^i((\bar\cX_L)_{\resbar L},\Q_l)$.
In particular, \smash{$t_1=\Tr\rho_\ell^{I_L}(\Frob_L)$} by the above.
Note that $t_0=1$ as $\bar\cX_L$ is connected, and $t_2$ is a multiple of $q=|\res L|$
\marginpar{Reference?}
as $\Het^2((\bar\cX_L)_{\resbar L},\Q_l)$ is the permutation module on the irreducible components of 
$\bar\cX_L$ twisted by $\Q_l(1)$.
%
%
%
Now, $|t_1|\le 2g\sqrt{q}$ because all eigenvalues of $\rho_\ell^{I_L}(\Frob_L)$ have absolute values 1 or $\sqrt{q}$
and $\dim\rho_\ell^{I_L}\le 2g$. 
Because $q\ge c>16g^2$, we have $2g\sqrt{q}<q/2$,~and~so
$$
  t_1 \bmod q \>\>\>=\>\>\> (t_0+t_2-\#\bar\cX_L(\res L)) \bmod q \>\>\>=\>\>\> (1-\#\bar\cX_L(\res L)) \bmod q
$$
recovers $t_1$ uniquely. This shows that $\#\bar\cX_L(\res L)$ determines $t_1=\Tr\rho_\ell^{I_L}(\Frob_L)$, 
and the collection of these for varying $L$ determines $\Het^1(X_{\bar K},\Q_l)$.
\end{proof}
\section*{An example}

%
%
%
%
We end with a numerical example of how Theorem \ref{X16} works in practice. 


Let $K=\Q_7$ and let $F/\Q_7$ be a wildly ramified Galois extension of degree 21 with (non-abelian) Galois group $G=C_7\rtimes C_3$, residue degree 3 and ramification degree 7. 
The character table of $G$ is as follows, after 
fixing a choice of a cube root of unity $\zeta_3$ and $\sqrt{-7}$ in $\C$:
\begin{table}[!h]
$$
\begin{array}{c|ccccc}
&\rm1& \tau & \tau^{-1} & \sigma & \sigma^{-1}\cr
\hline
  \triv    &1&1&1&1&1\cr
  \chi     &1&\zeta_3^2&\zeta_3&1&1\cr
  \bar\chi &1&\zeta_3&\zeta_3^2&1&1\cr
  \rho_{1} &3&0&0&\frac{-1-\sqrt{-7}}{2}&\frac{-1+\sqrt{-7}}{2}\cr
  \rho_{2} &3&0&0&\frac{-1+\sqrt{-7}}{2}&\frac{-1-\sqrt{-7}}{2}\cr
\end{array}
$$
\caption{Character table of $G=C_7{\rtimes}C_3$}
\end{table}

Consider a genus 3 curve $X$ whose Jacobian has bad reduction over $\Q_7$, but $X$ acquires good reduction over $F$.
Our goal is to determine the Galois representation $\rho=\Het^1(X_{\bar K},\Ql)\tensor_{\Ql}\C$.
By the N\'eron-Ogg-Shafarevich criterion (or just the proper smooth base change theorem), 
$I_K$ acts on $V$ through $I_{F/\Q_7}=C_7$. 
By Theorem \ref{X16}(i), and the fact that the characteristic polynomials of inertia elements have 
rational coefficients by \cite[Thm. 2]{ST},
$$
  \rho \iso (\rho_1\otimes \Psi_1) \oplus (\rho_2\otimes \Psi_2),
$$
for some unramified characters $\Psi_1, \Psi_2$. 


Theorem \ref{X16} lets us determine the $\Psi_i$ as follows. 
Let $\tilde{F}$ be the degree 7 unramified extension of $F$. For each element $g\in I_{F/\Q_7}=C_7$ let $\tilde{g}\in\Gal(\tilde{F}/\Q_7)$ be the unique element that projects to $g\in\Gal(F/\Q_7)$ and acts as $\Frob_{\Q_7}^3$ on the residue field. Write $F_{g}={\tilde F}^{\langle\tilde g\rangle}$ for the corresponding degree 21 extensions of $\Q_7$. Note that $\Jac X$ has good reduction over $F_g$, since $\tilde{F}/F_g$ is unramified.

By the theorem,
$$
  \Psi_i(\Frob_{\Q_7}^3) = \frac{1}{|C_7|} \frac{1}{\langle\Res_{C_7}\rho_i,\Res_{C_7}\rho_i\rangle} \sum_{g\in I_{F/\Q_7}} \overline{\Tr\rho_i(g)}\Tr\rho(\Frob_{F_g}) =
$$
$$
 = \frac{1}{21}\sum_{g\in I_{F/\Q_7}} \overline{\Tr\rho_i(g)}\cdot (7^3+1-\#\overline{\cX_{F_g}}(\F_{7^3})),
$$
where $\cX_{F_g}$ is a regular model of $X$ over $F_g$,
by the Grothendieck-Lefschetz trace formula (as in proof of Theorem \ref{abmain2} (iii)). 
Moreover, $\Psi_i(\Frob_{\Q_7})$ can be taken to be any cube root of this value, as the discrepancy vanishes when taking the tensor product with $\rho_i$.

As a particular example, let $F/\Q_7$ be the splitting field of
$$
  f(x)=x^7+21x^6+7,
$$
taken from \cite{JR}.
Let $\alpha$ be a root of $f$ in $F$; it is a uniformiser of $F$. 
Let $\zeta\in F$ be a primitive 18th root of unity. 
One checks that the roots $\alpha_j$ of $f$ have expansions
$$
  \alpha_j = \alpha + j\zeta\alpha^2 + O(\alpha^3), \qquad j=0,...,6.
$$
The generators of $G$ are then determined by
$$
  \sigma\colon \alpha_j\mapsto \alpha_{j+1} \text{ (order 7)}, 
  \qquad \tau\colon\alpha_j \mapsto \alpha_{2j} \text{ (order 3)},
$$
with indices taken modulo 7.
%

%

%


Now take $X: y^2=f(x)$. It acquires good reduction over $\Q_7(\alpha)$, and using Magma \cite{Magma} 
we find that  
$$
 \#\overline{\cX}_{F_{\id}}(\F_{7^3})= 7^3+1 = 344, 
$$
$$
 \#\overline{\cX}_{F_{\sigma}}(\F_{7^3})= \#\overline{\cX}_{F_{\sigma^2}}(\F_{7^3})= \#\overline{\cX}_{F_{\sigma^4}}(\F_{7^3})= 7^3-7^2+1 = 295,
$$
$$
 \#\overline{\cX}_{F_{\sigma^3}}(\F_{7^3})=  \#\overline{\cX}_{F_{\sigma^5}}(\F_{7^3})=  \#\overline{\cX}_{F_{\sigma^6}}(\F_{7^3})= 7^3+7^2+1 = 393,
$$
which leads to $\Psi_1(\Frob_{\Q_7})=\sqrt{-7}$ and  $\Psi_2(\Frob_{\Q_7})=-\sqrt{-7}$.

Conversely, let $X': y^2=x^7 + 420x^6 - 245x^3 + 1225x^2 - 833x + 189$. 
(The right-hand side is the minimal polynomial of $-\alpha^2-\alpha$.)
This curves acquires good reduction over the 
same fields as $X$, but has the opposite point counts (that is, 295 and 393 are interchanged), 
so that here $\Psi_1(\Frob_{\Q_7})=-\sqrt{-7}$ and $\Psi_2(\Frob_{\Q_7})=\sqrt{-7}$.

What makes the example interesting is that the two curves have the same point counts over all extensions 
of $\Q_7$ of degree at most 20. In that sense our result here is best possible.

\begin{remark}
\label{oops}
The assumption $\vchar K=0$ in Theorems \ref{abmain}, \ref{genunifield}, \ref{abmain2} and Corollary \ref{genunifield2} is necessary,
as we will now explain.  

First, suppose $M/K$ is any ramified separable extension of 
prime degree $p>2$. It is obtained by adjoining the root of some Eisenstein polynomial $f(x)$. Consider
the hyperelliptic Jacobian of dimension $g=(p-1)/2$
$$
  A=\Jac C, \qquad C: y^2=f(x).
$$
In the terminology of \cite{M2D2}, $C$ has one proper cluster of size $p$ 
(because $I_K$ acts transitively on roots of $f(x)$ and $\deg f=p$ is prime).
By \cite[Thm 1.20]{M2D2}, 
$$
  \rho_l=\Het^1(A_{\bar K},\Ql) \iso \gamma\tensor(\pi\ominus\triv),
$$
as an $I_K$-representation, where $\gamma$ is some tame character and $\pi$ is the permutation character 
of $I_K$ on the roots of $f$. 

Now, suppose $\vchar K=p$. 
By \cite{Kr}, there are extensions $M/K$ as above with arbitrary large valuation of the discriminant. 
Therefore, by the conductor-discriminant formula and the definition of the conductor, for every $n\ge 1$ there 
is a separable polynomial $f(x)\in K[x]$ of degree $p$ such that the higher ramification group $G_n\normal I_K$ 
acts non-trivially on the roots of $f(x)$. 

Let $K_g/K$ be any finite Galois extension. Then $G_n$ acts trivially on $K_g$ for some~$n$. However, the above
construction produces an abelian variety of dimension $g$ for which 
$\rho_l^{\smash{I_{K_g}}}\subset \rho_l^{G_n}=0$. 
Therefore $A$ has non-semistable reduction over $K_g$ and $\Tr\rho^{I_L}_\ell(\Frob_L)=0$ for every subfield
$K\subset L\subset K_g$. In other words, 
Theorems \ref{abmain}, \ref{genunifield}, \ref{abmain2} and Corollary \ref{genunifield2} fail over function fields.
\end{remark}

\subsection*{Acknowledgements}
We would like to thank 
the Warwick Mathematics Institute and 
King's College London, where parts 
of this research were carried out, 
Toby Gee and 
Aurel Page for helpful discussions,
and the referee for carefully reading the paper and for making us realise that Theorem \ref{abmain} 
fails over function fields.
This research was partially is supported by EPSRC grants EP/M016838/1 and EP/M016846/1 
`Arithmetic of hyperelliptic curves'. 
The second author was supported by a Royal Society 
University Research Fellowship.


\end{document}